\documentclass[12pt,reqno]{amsart}
\usepackage{amssymb,a4wide, xy}
\usepackage{amsmath}
\usepackage{amsfonts}

\xyoption{all}

\numberwithin{equation}{section}
\newtheorem{theorem}{Theorem}[section]

\newtheorem{lemma}[theorem]{Lemma}

\newtheorem{prop}[theorem]{Proposition}

\newtheorem*{con7*}{Conjecture 7*}
\newtheorem{Property}[theorem]{Property}
\newtheorem{Example}[theorem]{Example}
\newtheorem*{Conjecture}{Conjecture}
\newtheorem*{Main Theorem}{Main Theorem}

\theoremstyle{definition}
\newtheorem{Definition}[theorem]{Definition}

\theoremstyle{remark}
\newtheorem*{Remark}{Remark}

\numberwithin{equation}{theorem}

\DeclareMathOperator{\wgd}{w.gl.dim}
\DeclareMathOperator{\Tor}{Tor}
\DeclareMathOperator{\Spec}{Spec}
\DeclareMathOperator{\Max}{Max}
\DeclareMathOperator{\Ker}{Ker}
\DeclareMathOperator{\Img}{Im}
\begin{document}

\title[Bazzoni-Glaz Conjecture]{Bazzoni-Glaz Conjecture}

\author{Guram Donadze \and Viji Z. Thomas}
 \address{G. Donadze\\ Department of Algebra, University of Santiago de Compostela,
15782, Spain.}
 \email {gdonad@gmail.com}
\address{V.Z. Thomas \\ School of Mathematics, Tata Institute of Fundamental Research,
Mumbai, Maharashtra 400005, India.}
\email{vthomas@math.tifr.res.in}
\begin{abstract} In \cite{BG}, Bazzoni and Glaz conjecture that the weak global
dimension of a Gaussian ring is $0,1$ or $\infty$. In this paper, we prove their
conjecture.
\end{abstract}

\maketitle

\section{Introduction}
In her Thesis \cite{T}, H. Tsang, a student of Kaplansky introduced Gaussian rings.
Noting that the content of a polynomial $f$ over a commutative ring $R$ is the ideal
$c(f)$ generated by the coefficients of $f$, we now define a Gaussian ring.

\begin{Definition}[Tsang, 1965 \cite{T}]
A polynomial $f\in R[x]$ is called Gaussian if $c(f)c(g)=c(fg)$ for all $g\in R[x]$.
The ring $R$ is called Gaussian if each polynomial in $R[x]$ is Gaussian.
\end{Definition}

In her thesis, Tsang determined conditions under which a polynomial is Gaussian. In particular, she proved

\begin{theorem}[Tsang 1965 \cite{T}]
Let $R$ be a commutative ring and let $f\in R[x]$ be a polynomial in one variable over $R$. If $c(f)$ is an invertible ideal, or more generally, locally a principal ideal, then $f$ is a Gaussian polynomial.
\end{theorem}

The converse of this statement has received considerable interest in the recent past and is related to the following conjecture of Kaplansky.
\begin{Conjecture}[Kaplansky]
Let $R$ be a commutative ring and let $f\in R[x]$ be a Gaussian polynomial. Then
$c(f)$ is an invertible or, at least, a locally principal ideal.
\end{Conjecture}

A number of authors contributed to the solution of the above conjecture of Kaplansky. D.D. Anderson and Kang (\cite{AK}) initiated the renewed interest in Kaplansky's conjecture; Glaz and Vasconcelos (\cite{GV}, \cite{GV1}) showed that the conjecture holds for integrally closed Noetherian domains; Heinzer and Huneke (\cite{HH}) proved that the conjecture holds for Noetherian domains; Loper and Roitman (\cite{LR}) solved the conjecture for all domains; Lucas (\cite{L}, \cite{L1}) extended this result to a partial answer for non-domains. In general, Kaplansky's conjecture is false (see \cite{GV}, \cite{GV1}).

  A Pr\"ufer (see Definition \ref{pru}) domain is a type of commutative ring that generalizes Dedekind domains in a non-Noetherian context, in particular, a Noetherian Pr\"ufer domain is a Dedekind domain. Among other things, H. Tsang (\cite{T}) proved that an integral domain is Gaussian if and only if it is Pr\"ufer, a result also proved independently by R. Gilmer in \cite{RG}. Thus Gaussian rings provide another class of rings extending the class of Pr\"ufer domains to rings with zero divisors. A homological generalization to rings $R$ of the Pr\"ufer property for domains, is the condition: $R$ has weak global dimension less or equal to one, $\wgd \leq 1$. A subfamily of this class of rings is the class of semihereditary rings. Recall that a ring is semihereditary if every finitely generated ideal of $R$ is projective. Semihereditary rings $R$ are precisely those rings $R$ with $\wgd\leq 1$ which are coherent. Domains of weak global dimension less or equal to 1, in particular, semihereditary domains, are Pr\"ufer domains. A lot of work and progress has been made in investigating Pr\"ufer like conditions in commutative rings in the last 15 years, this can be found in the survey article \cite{GS}. To give a homological characterization to several classes of rings has been a topic of research for several years, for example a ring $R$ is semi-simple if and only if the global dimension of $R$ is $0$ (\cite{SG1}). Similarly, $\wgd R=0$ if and only if $R$ is a Von-Nuemann regular ring (see \cite{H}). We note that it follows from Osofsky \cite{O} that arithmetical rings have weak global dimension at most one or $\infty$. In the same vein, the author of \cite{SG} is concerned with giving a homological characterization to Gaussian rings. She considers the question: For a Gaussian ring $R$, what possible values may $\wgd R$ take? In \cite{BG}, the authors consider five possible extensions of the Pr\"ufer domain notion to the case of commutative rings with zero divisors, two among which are Gaussian rings and rings with weak global dimension (see Definition \ref{weak}) at most one. They consider the problem of determining the possible values for the weak global dimension of a Gaussian ring. At the end of their paper, they make the following conjecture.
\begin{Conjecture}[Bazzoni-Glaz, 2007 \cite{BG}]
The weak global dimension of a Gaussian ring is either $0,1$ or $\infty$.
\end{Conjecture}
Our aim in this article is to give a proof of the above conjecture. The above conjecture is also listed as an open question in the recent survey article \cite{GS}. In a recent paper \cite{AJK}, the authors have validated the Bazzoni-Glaz conjecture
for the class of rings called fqp-rings. The class of fqp-rings fall strictly between the classes of arithmetical rings and Gaussian rings. In \cite{SG}, the author shows that the weak global dimension of a coherent Gaussian ring is either $\infty$ or at most one. She also shows that the weak global dimension of a Gaussian ring is at most one if and only if it is reduced. So to prove the conjecture it is enough to show that $\wgd R=\infty$ for all non-reduced Gaussian rings $R$. Since $\wgd R=\sup\{\wgd R_{\mathfrak{p}}\mid \mathfrak{p}\in \Spec (R)\}$, it is enough to prove the conjecture for non-reduced local Gaussian rings. For any non reduced local Gaussian ring $R$ with nilradical $\mathcal{N}$, either $(i)$ $\mathcal{N}$ is nilpotent or $(ii)$ $\mathcal{N}$ is not nilpotent. Except when $\mathcal{N}^2=0$, the authors of \cite{BG} prove that if $R$ satisfies $(i)$, then $\wgd R=\infty$. In this paper we prove that if $R$ satisfies $(ii)$, then $\wgd R=\infty$ (cf. Theorem \ref{main}). We also give a complete proof of $(i)$. Now we briefly describe the strategy of the proof. After localization at minimal prime ideal, $\mathcal{N}$, the maximal ideal and the nilradical of $R_{\mathcal{N}}$ coincide, let us denote it by $\mathcal{N'}$. If $\mathcal{N}^2\neq 0$, then $\mathcal{N'}\neq 0$. Hence to prove the Bazzoni-Glaz conjecture in this case, it suffices to show that the weak global dimension of a local Gaussian ring with the maximal ideal coinciding with the nilradical is infinite. If $\mathcal{N}^2=0$, then $\mathcal{N'}$ can be zero (see Example \ref{EG}). Hence we discuss this case separately in Section 6.

In Section 3, we consider some homological properties of local Gaussian rings. In
particular we consider local Gaussian rings $(R,\mathfrak{m})$ which are not fields,
with the property that each element of $\mathfrak{m}$ is a zero divisor. In this
case we prove that $\wgd R\geq 3$.

In  \cite[Section 6]{BG}, the authors consider local Gaussian rings $(R,
\mathfrak{m})$ such that the maximal ideal $\mathfrak{m}$ coincides with the
nilradical of $R$. With this set up in Section 4, we prove that if $\mathcal{N}$ is
not nilpotent, then $\wgd R=\infty$.

In Section 5, we prove the conjecture except the case $\mathcal{N}^2=0$, and in the last
section we prove it completely.

Throughout this paper, $R$ is a commutative ring with unit, $(R,\mathfrak{m})$ is a
local ring(not necessarily Noetherian) with unique maximal ideal $\mathfrak{m}$. For $R$-modules $M$ and $N$, we write $M\otimes N$ and $\Tor_*(M,N)$ instead of $M\otimes_R N$, and $\Tor_*^R(M, N)$, respectively. We denote the set of all prime ideals of $R$ by $\Spec(R)$ and the set of all maximal ideals by $\Max(R)$.

\section{Preliminary Results}

In this section we will recall some definitions and results that we will need in
later sections.

\begin{Definition}[\cite{SG1}]
Let $M$ be a module over $R$. The weak global dimension of $M$ (denoted by $\wgd_R(M)$)
is the minimum integer (if it exists) such that there is a resolution of $M$ by flat $R$
modules $0\to F_n\to \cdots\to F_1\to F_0\to M\to 0$. If no finite resolution by flat $R$
modules exists for $M$, then we set $\wgd_R(M)=\infty$.
\end{Definition}

Now we define the weak global dimension of a ring $R$ denoted as $\wgd(R)$. It is
also sometimes suggestively called as the $\Tor$-dimension.

\begin{Definition}[\cite{SG1}]\label{weak}
$\wgd(R)$=$\sup$\{$\wgd_R(M) \mid M$ is an $R$-module\}.
\end{Definition}

Recall that $\wgd(R)$=$\sup$\{$d \mid \Tor_d(M,N)\neq 0$ for some $R$-modules
$M,N$\}. The $\wgd(R)\leq 1$ if and only if every ideal of $R$ is flat, or
equivalently, if and only if every finitely generated ideal of $R$ is flat.

We now define a Pr\"ufer domain.

\begin{Definition}[Pr\"ufer 1932 \cite{P}]\label{pru}
An integral domain $R$ is called a Pr\"ufer domain if every finitely generated non-zero ideal of $R$ is invertible.
\end{Definition}

L. Fuchs introduced the class of arithmetical rings in \cite{F}.

\begin{Definition}[Fuchs 1949 \cite{F}]
A ring $R$ is arithmetical if the lattice of the ideals of $R$ is distributive.
\end{Definition}

 In \cite{J}, the author characterized arithmetical rings by the property that in every localization at a maximal ideal, the lattice of the ideals is linearly ordered by inclusion. Hence in a local arithmetical ring, the lattice of the ideals is linearly ordered by inclusion. Thus local arithmetical rings provide another class of rings extending the class of valuation domains to rings with zero-divisors.

The next theorem appears in Tsang's (see \cite{T}) unpublished thesis.

\begin{theorem}[Tsang 1965 \cite{T}]\label{thm2.1}
Let $R$ be a Gaussian ring. If $R$ is local, then
\begin{itemize}
  \item[(i)] $R$ is Gaussian if and only if $R_{\mathfrak{m}}$ is Gaussian for all
$\mathfrak{m}\in \Max(R)$;
   \item[(ii)] $R$ is Gaussian if and only if $R_{\mathfrak{p}}$ is Gaussian for all
$\mathfrak{p}\in \Spec(R)$;
   \item[(iii)] the prime ideals of $R$ are linearly ordered under inclusion; and
 \item[(iv)] the nilradical of $R$ is the unique minimal prime ideal of $R$.
\end{itemize}
\end{theorem}

We will need several equivalent characterizations of local Gaussian rings, which we
now state.

\begin{theorem}[\cite{T}, \cite{L}]\label{thm2.2}
Let $(R,\mathfrak{m})$ be a local ring with maximal ideal $\mathfrak{m}$. The
following conditions are equivalent.
\begin{itemize}
  \item[(i)] $R$ is a Gaussian ring;
   \item[(ii)] If $I$ is a finitely generated ideal of $R$ and $(0:I)$ is the
annihilator of $I$, then $I/{I\cap(0:I)}$ is a cyclic $R$-module;
   \item[(iii)] Condition $(ii)$ for two generated ideals;
 \item[(iv)] For any two elements $a,b\in R$, the following two properties hold:
 \begin{itemize}
   \item[(a)] $(a,b)^2=(a^2)$ or $(b^2)$;
  \item[(b)] If $(a,b)^2=(a^2)$ and $ab=0$, then $b^2=0$.
 \end{itemize}
 \item[(v)] If $I=(a_1,a_2,\ldots,a_n)$ is a finitely generated ideal of $R$, then
$I^2=(a_{i}^2)$ for some $1\leq i\leq n$.
\end{itemize}
\end{theorem}

The implication $(iv)\Rightarrow (i)$ was noted by Lucas in \cite{L} and the rest of
Theorem \ref{thm2.2} was proved by Tsang in \cite{T}. The next two results can be
found in \cite{BG}.

\begin{theorem}[Bazzoni-Glaz 2007 \cite{BG}]\label{thm2.3}
Let $(R,\mathfrak{m})$ be a local Gaussian ring and let $D=\{x\in R \mid x^2=0\}$.
The following hold:
\begin{itemize}
  \item[(i)] $D$ is an ideal of $R$, $D^2=0$, and $R/D$ is an arithmetical ring;
   \item[(ii)] For every $a\in R$, $(0:a)$ and $D$ are comparable and $D\subseteq
Ra+(0:a)$;
   \item[(iii)] If $a\in \mathfrak{m}\setminus D$, then $(0:a)\subseteq D$;
 \item[(iv)] Let $\mathfrak{m}$ be the nilradical of $R$. If $\mathfrak{m}$ is not
nilpotent, then $\mathfrak{m}=\mathfrak{m^2}+D$ and
$\mathfrak{m^2}=\mathfrak{m^3}$.
 \end{itemize}
\end{theorem}

\begin{prop}[Bazzoni-Glaz 2007 \cite{BG}]\label{prop2.1}
Let $(R,\mathfrak{m})$ be a local Gaussian ring. If $\mathfrak{m}$ is non-zero and
nilpotent, then $\wgd_R  \mathfrak{m}=\infty$.
\end{prop}

\section{Some results on local Gaussian rings}

Throughout this section (as well as in the sequel) $D$ is supposed to be as in Theorem \ref{thm2.3}.
It is well known that if the $\wgd_R(M)=n$, then there exists a cyclic $R$-module, say
$R/I$ such that $\Tor_n(R/I, M)\neq 0$. In the next lemma, we show that this cyclic module can
be chosen with some additional properties.

\begin{lemma}\label{lemma2}
Let $R$ be local Gaussian ring and $M$ be a module over $R$ with $\wgd_R(M)=n$. Let $I$ be an
ideal of $R$. If $\Tor_n(R/I, M)\neq 0$, then there exists an ideal $J\subset R$ such that
$\Tor_n(R/J, M)\neq 0$ and $J\supseteq I+D$.
\end{lemma}
\begin{proof} Suppose the lemma is not true. Then we will prove that the natural projection $R/I \to
R/(I+x_1R+\cdots +x_mR)$ induces an inclusion
\begin{equation}\label{tor1}
 \Tor_n (R/I, M)\hookrightarrow \Tor_n (R/(I+x_1R+\cdots +x_mR), M)
\end{equation}
 for any finite subset $\{x_1, \dots, x_m\}\subset D$. Towards that end, set $I_0=I$ and define
 $I_p$ inductively as $I_p=I_{p-1}+x_p R$ for all $1\leq p\leq m$.
 We have the following short exact sequence
\[
 0\to (x_pR+I_{p-1})/I_{p-1} \to R/I_{p-1} \to R/I_p \to 0
\]
 for all $1\leq p\leq m$. The homomorphism $f: R \to (x_pR+I_{p-1})/I_{p-1}$ defined
by $f(r)=rx_p+I_{p-1}$ for all $r\in R$ induces an isomorphism $R/\Ker f \cong (x_p
R+I_{p-1})/I_{p-1}$. Furthermore we have that $\Ker f\supset I_{p-1}+ (0:x_p)
\supset I+D$. If $\Tor_n ((x_pR+I_{p-1})/I_{p-1}, M)\neq 0$, then the lemma is true
with $J=\Ker f$. So assume that $\Tor_n ((x_pR+I_{p-1})/I_{p-1}, M)= 0$. In this
case the natural projection $R/I_{p-1} \to R/I_p$ induces an inclusion
$\Tor_n (R/I_{p-1}, M)\hookrightarrow \Tor_n (R/I_p, M)$ for all $1\leq p\leq m$,
proving (\ref{tor1}).

 Now let $\mathcal{X}$ denote the following class of ideals: $J\in \mathcal{X}$ iff
$J\subset D$ and $J$ is finitely generated. Then
\begin{equation}\label{directlim}
\varinjlim_{J\in \mathcal{X}} \Tor_n(R/(I+J),M)=\Tor_n(R/(I+D),M)
\end{equation}

Using (\ref{tor1}) and (\ref{directlim}), we obtain an inclusion $\Tor_n (R/I,
M)\hookrightarrow \Tor_n (R/(I+D), M)$. Thus $\Tor_n (R/(I+D), M)\neq 0$ and the
lemma is proved.
\end{proof}

The next lemma is an immediate consequence of the long exact sequence of $\Tor$
groups applied to the given short exact sequence. We note it here for the readers
convenience.

\begin{lemma}\label{lemma3}
Let $R$ be a commutative (not necessarily local Gaussian) ring. Let $M, M_1$ and $M_2$
be $R$-modules and $f:M_1\to M_2$
be an injective homomorphism. If the $\wgd_R(M)=n$, then $f_*:\Tor_n(M_1, M) \to
\Tor_n(M_2, M)$ is also injective.
\end{lemma}
\begin{proof} The proof is a direct consequence of the long exact sequence of $\Tor$
groups applied to the given short exact sequence and the fact that
$\Tor_{n+1} (M_2/M_1, M)=0$.
\end{proof}

The next lemma is very useful for us in the sequel.

\begin{lemma}\label{lemma4}
Let $(R,\mathfrak{m})$ be a local Gaussian ring and $M$ be a module over $R$. If \newline
$\wgd_R(M)=n\geq 1$, then $\Tor_n (R/D, M )=0$.
\end{lemma}
\begin{proof} Suppose the lemma is not true. Consider a free resolution of $M$:
\newline $\cdots \xrightarrow{\partial_{n+2}} R^{X_{n+1}}\xrightarrow{\partial_{n+1}}
R^{X_n}\xrightarrow{\partial_{n}}
\cdots \xrightarrow{\partial_1} R^{X_0}\xrightarrow{\partial_0} M$, where $X_i$ are
sets. By assumption
$\Tor_n (R/D, M)=\Ker(\overline{\partial_n})/ \Img(\overline{\partial_{n+1}})\neq
0$, where $\overline{\partial_i}$ is the natural homomorphism
$\overline{\partial_i}:(R/D)^{X_i}\to (R/D)^{X_{i-1}}$ obtained after tensoring the
above resolution by $R/D$ for all $i\in \mathbb{N}$. Since $\Tor_n (R/D, M)\neq 0$,
there exists a $\overline{w}\in \Ker(\overline{\partial_n})$ such that
$\overline{w}\notin \Img(\overline{\partial_{n+1}})$. Let $w$ be the representative
of $\overline{w}$ in $R^{X_n}$. Hence $\partial_n(w)\in D^{X_{n-1}}$. Let
$\lambda_1, \dots,\lambda_m\in D$ be the finitely many non-zero entries of
$\partial_n(w)$. Now we consider two cases.

{\bf Case 1}. There exists an $a\in \mathfrak{m}\setminus D$ such that
$a\lambda_j=0$ for all $1\leq j\leq m$.\\
Define a homomorphism $f:R/D \to R/a D$ which is multiplication by $a$. Using
Theorem \ref{thm2.3}$(iii)$, it follows that $(0:a)\subset D$. This gives the
injectivity of $f$. Therefore by Lemma \ref{lemma3}, $f_*:\Tor_n (R/D, M) \to \Tor_n
(R/(aD), M)$ is injective and hence $f_*(\overline{w})\neq 0$. It is easy to verify
that $aw$ is a representative of $f_*(\overline{w})$ in $R^{X_n}$.
Since $a\in (0:\lambda_j)$ for all $1\leq j\leq m$, we obtain that
$\partial_n(aw)=a\partial_n(w)=0$. This would imply that $f_*(\overline{w})=0$, a
contradiction.

{\bf Case 2}. For all $a\in R\setminus D$ at least one $a\lambda_j\neq 0$.\\
We have an injective homomorphism $g: R/D\to R^m$ defined by $g(1)=(\lambda_1, \dots
, \lambda_m)$. By Lemma \ref{lemma3}, the induced homomorphism
$g_*:\Tor_n (R/D, M) \to \Tor_n (R^m,M)$ is injective. This is a contradiction as
$\Tor_n (R^m,M)=0$.
\end{proof}

\begin{lemma}\label{llemma}
Let $(R,\mathfrak{m})$ be a local Gaussian ring and $M$ be a module over $R$. If \newline
$\wgd_R(M)=n\geq 1$, then there exists an $a\in \mathfrak{m}\setminus D$ such that $\Tor_n(R/(aR+D), M)\neq 0$.
\end{lemma}
\begin{proof} There exists an ideal $I\subset R$ such that $\Tor_n(R/I, M)\neq 0$.
Without loss of generality one can assume that $D\subset I$
(see Lemma \ref{lemma2}). Using Lemma \ref{lemma4}, we obtain that $I\neq D$. Let $\mathcal{X}$
denote the following class of ideals: $J\in \mathcal{X}$ iff $J\subset I$ and $J$ is finitely
generated. Since $\Tor_n(R/I, D)= \varinjlim_{J\in \mathcal{X}}\Tor_n(R/ J, M)\neq 0$,
there exist $a_1,\dots, a_m\in I$ such that $\Tor_n (R/ (a_1R+\cdots +a_mR+D), M)\neq 0$.
By Theorem \ref{thm2.3}$(i)$, $R/D$ is a local arithmetical ring. Hence there exists an $i$
such that $a_1R+\cdots +a_mR+D=a_iR+D$ for some $1\leq i\leq m$.
\end{proof}

Let $(R,\mathfrak{m})$ be a local ring such that each element of
$\mathfrak{m}$ is a zero divisor. If $b\in \mathfrak{m}$, then there exists an element
$r_b\in \mathfrak{m}$ such that $br_b=0$. In general, the element $r_b$ depends on $b$.
In the next lemma, we show that if $(R,\mathfrak{m})$ is a local Gaussian ring, then a
slightly stronger result is true. In particular, we show that for finitely many elements
in $\mathfrak{m}$, there exists a single element in $\mathfrak{m}$ that annihilates all of them.
\begin{lemma}\label{singleann}
Let $(R,\mathfrak{m})$ be a local Gaussian ring such that each element of
$\mathfrak{m}$ is a zero divisor. If $\lambda_1, \dots, \lambda_n\in \mathfrak{m}$,
then there exists a non-trivial element $a\in \mathfrak{m}$ such that $a\lambda_j=0$
for all $1\leq j\leq n$.
\end{lemma}

\begin{proof}
We divide the proof into two cases.

{\bf Case 1}: $\lambda_1,\dots,\lambda_n\in D$.

 For $1\leq i\leq n$, choose any $\lambda_i\in D\setminus 0$. Using Theorem \ref{thm2.3}$(i)$, the desired result follows.

{\bf Case 2}:  There exist $j\in \{1,2,\ldots,n\}$ such that $\lambda_j\notin D$.

By Theorem \ref{thm2.3}$(iii)$, it follows that $(0:\lambda_j)\subseteq D$. Set
$I=(\lambda_1,\ldots, \lambda_n)$. So $(0:I)\subset (0:\lambda_j)\subseteq D$. Using
Theorem \ref{thm2.2}$(ii)$, we obtain that $I/{I\cap (0:I)}$ is a cyclic $R$-module,
say its generator is $\lambda$. Hence we can write $\lambda_i=r_i\lambda+d_i$ for
all $1\leq i\leq n$, where $d_i\in I\cap (0:I)$ and $r_i\in R$. Observe that
$\lambda\in \mathfrak{m}\setminus D$. Choose any $d\in (0:\lambda)\setminus 0$.
Using Theorem \ref{thm2.3}$(iii)$, it follows that $d\in D$. Multiplying the
equation expressing $\lambda_i$ in terms of $\lambda$ with $d$, we obtain
$d\lambda_i=dr_i\lambda + dd_i$ for all $1\leq i\leq n$. Using  Theorem
\ref{thm2.3}$(i)$, we obtain that $dd_i=0$. Thus $d\lambda_i=0$ for all $1\leq i\leq
n$.
\end{proof}

\begin{lemma}\label{lemma5}
Let $(R,\mathfrak{m})$ be a local Gaussian ring and $M$ be a module over $R$. If each element of
$\mathfrak{m}$ is a zero divisor and $\wgd_R(M)=n\geq 1$, then $\Tor_n (R/\mathfrak{m}, M )=0$.
\end{lemma}
\begin{proof} The proof of this Lemma follows by substituting $\mathfrak{m}$ for $D$
in Lemma \ref{lemma4}. As a result of Lemma \ref{singleann}, the proof of lemma
\ref{lemma5} falls under Case 1 of Lemma \ref{lemma4}.
\end{proof}

The next result seems to be restricted but it is useful for our further purposes.

\begin{prop}\label{prop1.1}
Let $(R,\mathfrak{m})$ be a local Gaussian ring. If $\mathfrak{m}\neq 0$ and each
element of $\mathfrak{m}$ is a zero divisor, then $\wgd (R)\geq 3$.
\end{prop}
\begin{proof} If $\mathfrak{m} =D$, then Proposition \ref{prop2.1} implies that
$\wgd (R) = \infty$. If $\mathfrak{m}\neq D$, then take any $x\in
\mathfrak{m}\setminus D$ and consider
the following resolution of $R/xR$:
\[
0\to (0:x)\to R\xrightarrow{x_m}R\xrightarrow{\pi}R/xR ,
\]
where $\pi$ is the natural projection and $x_m$ is multiplication by $x$. If $\wgd
(R)<3$, then $(0:x)$ must be flat. Thus
it suffices to show that $(0:x)$ is not flat. We will use the fact that if $M$ is a
flat $R$ module then $I\otimes M=IM$ for all ideals $I\subset R$. Set $I=xR$ and
$M=(0:x)$ and observe that $IM=0$. Hence it suffices to show that $I\otimes
(0:x)\neq 0$. Since $x\notin D$, Theorem \ref{thm2.3}$(iii)$ implies that
$(0:x)\subset D$. Define a homomorphism $\theta : I\otimes (0:x)\to (0:x)$ as
follows: if $a\in I$ and $b\in (0:x)$, then set $\theta(a\otimes b)=rb$, where $r\in
R$
is such that $a=xr$. If there is another $r'\in R$ such that $a=xr'$, then
$(r-r')\in (0:x)$ which implies that $(r-r')b=0$.
Taking into account the last remark, it is easy to check that $\theta$ is well
defined. Moreover the homomorphism
$\theta':(0:x)\to I\otimes (0:x)$ defined by $\theta'(c)=x\otimes c$ for all $c\in
(0:x)$ is an inverse of $\theta$. Hence we have an isomorphism
$\theta:I\otimes (0:x)\cong (0:x)$ which shows that $I\otimes (0:x)\neq 0$, proving
that $(0:x)$ is not flat.
\end{proof}

Let $R$ be a local Gaussian ring which admits the following property:
\begin{Property}\label{hypo1}
For all $x\in D\setminus 0$, $(0:x)$ is not cyclic modulo $D$. In other words there
is no $a\in R\setminus D$ such that $(0:x)=aR+D$.
\end{Property}

By Theorem \ref{thm2.2}, we know that if $I=(a_1,a_2,\cdots,a_n)$ is a finitely generated ideal of a local Gaussian ring, then $I^2=(a_i^2)$ for some $i$, where $1\leq i\leq n$. But for this particular $i$, $a_i$ could have the property that $a_i^4=0$. In the next lemma, we show that if $(R,\mathfrak{m})$ is a local Gaussian ring with some additional hypothesis, then the square of any finitely generated proper ideal of $R$ is contained in the square of an element, say $x\in \mathfrak{m}$ with $x$ having the property that $x^4\neq 0$. We also show that $\mathfrak{m}^2$ is flat.

\begin{lemma}\label{flat} Let $(R,\mathfrak{m})$ be a local Gaussian ring such that
each element of $\mathfrak{m}$ is a zero divisor. If $R$ admits Property
(\ref{hypo1}) and $\mathfrak{m}\neq D$ , then

\begin{itemize}
\item[(i)] $\mathfrak{m}=\mathfrak{m}^2+D$;
\item[(ii)] for any finitely generated ideal $J\subset \mathfrak{m}$ there exist
$x\in \mathfrak{m}$ such that $J^2\subset x^2R$ and $x^2\notin D$;
\item[(iii)] $\mathfrak{m}^2$ is flat.
\end{itemize}
\end{lemma}

\begin{proof} (i): Let $a\in \mathfrak{m}\setminus D$. Since every element of
$\mathfrak{m}$ is a zero divisor, there exists $x\in D\setminus 0$ such that $ax=0$.
By Property (\ref{hypo1}), $(0:x)\neq aR+D$. So there exists some $b\in
\mathfrak{m}$ such that $b\in (0:x)$ and $b\notin aR+D$. Theorem \ref{thm2.3}$(i)$
implies that $R/D$ is a local arithmetical ring. So $a\in b R+D$ and hence $a=b r +
d$ for some $r \in R$ and $d\in D$. Moreover $b\notin aR+D$ which implies that $r$
is not a unit and hence $r\in \mathfrak{m}$. Thus $a\in \mathfrak{m}^2+D$.

(ii): First we will show that if $x^2\in D$ for all $x\in \mathfrak{m}$, then
$\mathfrak{m^2}\subset D$. Towards that end let $z\in \mathfrak{m^2}$. Such a $z$ is
of the form $z=\sum_{i=1}^n x_iy_i$, where $x_i,y_i\in m$ for all $1\leq i\leq n$.
Using Theorem \ref{thm2.2}$(iv)$, it follows that
$(x_i,y_i)^2=(x_i^2)\;\text{or}\;(y_i^2)$ for all $1\leq i\leq n$. This shows that
$x_iy_i\in D$ for all $1\leq i\leq n$. Recalling that $D$ is an ideal of $R$, it
follows that $z\in D$. Hence we have proved that $\mathfrak{m^2}\subset D$. By $(i)$,
this would imply that $\mathfrak{m}=D$, a contradiction. Thus there exists an $x\in
m$ such that $x^2\notin D$. By Theorem \ref{thm2.2}$(v)$, for any finitely generated
ideal $J$ we have $J^2=y^2R$ for some $y\in J$. If $y^2\notin D$, then we are done.
If $y^2\in D$, choose any $x\in \mathfrak{m}$ with $x^2\notin D$ and observe that
$y^2\in x^2R$. Thus $J^2\subset x^2 R$.

(iii): To prove that $\mathfrak{m}^2$ is flat over $R$, we show that for any ideal
$I\subset R$, the natural homomorphism $f: I\otimes \mathfrak{m}^2 \to
\mathfrak{m}^2$ is injective. Assume that $w\in I\otimes \mathfrak{m}^2$ is such
that $f(w)=0$. Set $w = \sum_{i=1}^k  z_i\otimes x_i y_i$, where $z_i\in I$ and
$x_i,y_i\in \mathfrak{m}$. By (ii), there exist $x\in \mathfrak{m}$ such that
$x^2\notin D$ and $x_iy_i\in x^2 R$ for all $1\leq i\leq n$. Put $x_iy_i=x^2r_i$,
where $r_i\in R$. Then $w=z\otimes x^2$, where $z =\sum_{i=1}^k z_ir_i \in I$.
Hence $f(z \otimes x^2)=0 \Leftrightarrow zx^2=0$. If $z=0$, then $w=0$ and the proof
is finished.
So assume that $z\neq 0$. Using Theorem \ref{thm2.3}$(iii)$, we obtain that
$(0:a)\subseteq D$ for all $a\in \mathfrak{m}\setminus D$. Since $x^2\in
\mathfrak{m}\setminus D$, it follows that $z\in D$ and either $zx = 0$ or $zx\neq
0$. If $zx = 0$, then $z\in D\setminus 0$ and $x\in (0:z)$.
It follows from Property (\ref{hypo1}) that $(0:z)\neq xR+D$. So there exists $y\in
\mathfrak{m}$ such that $y\in (0:z)$ and $y\notin xR+D$. By Theorem
\ref{thm2.3}$(i)$, we obtain that $R/D$ is a local arithmetical ring. Hence
$(y)\not\subset (x)$. So $x=cy+d'$, where $c\in \mathfrak{m}$ and $d'\in D$.
Computing $w$, we obtain
\[
w = z\otimes x^2 = z\otimes (cy+d')^2=z\otimes (c^2y^2+2cyd'+d'^2)=zy^2\otimes
c^2+zd'\otimes 2cy+z\otimes d'^2 .
\]
Noting that $d'^2,zd'\in D^2=0$ and that $zy^2=0$, we obtain $w=0$. If $zx\neq 0$,
we have $zx\in D\setminus 0$ and $x\in (0:zx)$. By (\ref{hypo1}), there exists
$h\in \mathfrak{m}$ such that $h\in (0:zx)$ and $h\notin xR+D$. Using the same
argument as above, there exists an $a \in \mathfrak{m}$ such that $x=ah+d''$.
Observing that $zd'',d''^2\in D^2$ we obtain that $w = z\otimes x^2=z\otimes
(ah+d'')^2=z\otimes a^2h^2$. Furthermore, by (i) we can write $a=b+d$ where $b\in
\mathfrak{m}^2$ and $d\in D$. Therefore
\begin{equation}\label{tensor}
w=z\otimes (ah^2(b+d))= z\otimes (ah^2b)+z\otimes (ah^2d) =(zah^2)\otimes
b+(zd)\otimes (ah^2).
\end{equation}
Substituting $0=zd\in D^2$ and $ah=x-d''$ in (\ref{tensor}) and recalling that $h\in
(0:zx)$, we obtain $w=(zxh)\otimes b-zhd''\otimes b=0$.
\end{proof}

\begin{lemma}\label{lemma1.2}
Let $(R,\mathfrak{m})$ be a local Gaussian ring such that
each element of $\mathfrak{m}$ is a zero divisor. If $R$ admits Property
(\ref{hypo1}), $\mathfrak{m}\neq D$ and $\wgd R<\infty$, then $\mathfrak{m}$ is flat.
\end{lemma}
\begin{proof} If $\mathfrak{m} =\mathfrak{m^2}$, then the result follows by Lemma \ref{flat}.
So assume that $\mathfrak{m} \neq \mathfrak{m^2}$.
By assumption, the $\wgd_R(R/\mathfrak{m})=n<\infty$. Hence there exists an $R$-module $M$ such
that $\Tor_n(R/\mathfrak{m}, M)\neq 0$. If we show that $n\leq 1$,
then the lemma will be proved. Suppose that $n\geq 2$. Using Lemma \ref{flat}$(iii)$, we obtain that $\Tor_{n}(R / \mathfrak{m^2}, M)=0$. Now consider the short exact sequence
$0\to \mathfrak{m} / \mathfrak{m^2} \to R / \mathfrak{m^2}\to R / \mathfrak{m} \to 0$ and
the following segment of the corresponding long exact sequence of $\Tor$
groups
\[
 \Tor_{n+1}(R / \mathfrak{m} , M)\to \Tor_{n}(\mathfrak{m}/ \mathfrak{m^2}, M)\to
\Tor_{n}(R / \mathfrak{m^2}, M) .
\]
The above sequence implies that $\Tor_{n}(\mathfrak{m}/ \mathfrak{m^2}, M)=0$.
Observing that $\mathfrak{m} / \mathfrak{m^2}$ is a vector space over $R/\mathfrak{m}$,
we obtain that $\mathfrak{m} /\mathfrak{m^2} = \bigoplus R/\mathfrak{m}$. Hence
$0=\Tor_{n}(\mathfrak{m} / \mathfrak{m^2}, M)=\Tor_{n}(\bigoplus
R/\mathfrak{m}, M)=\bigoplus \Tor_{n}(R/\mathfrak{m}, M)$, a contradiction.
\end{proof}

\section{Local Gaussian rings with nilradical being the maximal ideal}

The idea of the next lemma is taken from \cite{O}, but we give a more general result and with a slightly different
proof.

\begin{lemma}\label{lemma6}
Let $(R,\mathfrak{m})$ be a local arithmetical ring with nilradical $\mathfrak{m}$.
For any $x\in \mathfrak{m}\setminus 0$, if $(0:x)=I$ then $(0:I)=(x)$.
\end{lemma}
\begin{proof}  Clearly $(x)\subseteq (0:I)$. We want to show that $(0:I)\subseteq
(x)$. Towards that end assume that there exists a $z\in (0:I)$ such that $z\notin
(x)$. Recalling that the ideals in a local arithmetical ring are linearly ordered
under inclusion, we obtain that $x=\lambda z$ where $\lambda\in \mathfrak{m}$. Hence
$\lambda\notin I$ which implies that $I\subset (\lambda)$. By induction on $k$, we
will show that $I\subset (\lambda^k)$ for all $k\in \mathbb{N}$. The case $k=1$ is
obvious. Let $b\in I$ be arbitrary. Since $I\subset (\lambda)$ there exists a $t\in
\mathfrak{m}$ such that $b=\lambda t$. Notice that we have $0=zb=z\lambda t=xt$.
Hence $t\in I=(0:x)$. By the induction hypothesis, $I\subset (\lambda^k)$. So
$t=\lambda^kt_1$ where $t_1\in \mathfrak{m}$. Hence $b=\lambda^{k+1} t_1$, where
$t_1\in \mathfrak{m}$. Thus $I\subset (\lambda^{k+1})$ for all $k\in \mathbb{N}$.
Since $\lambda$ is nilpotent, we obtain that $I=0$, a contradiction.
\end{proof}

In what follows let $R'=R/D$ and $\mathfrak{m'}= \mathfrak{m}/D$. Recall that if $R$
is a local Gaussian ring, then $R'$ is a local arithmetical ring by Theorem
\ref{thm2.3}$(i)$.

\begin{lemma}\label{lemma7}
Let $(R,\mathfrak{m})$ be a local Gaussian ring with nilradical $\mathfrak{m}$ and let $M$
be an $R$-module. If $\wgd_R (M)=n\geq 1$, then the following conditions hold.
\begin{itemize}
  \item[(i)] there is a non trivial element $x\in \mathfrak{m'}$ such that
$\Tor_n(R'/xR', M)\neq 0$;
 \item[(ii)] for any non trivial element $z\in \mathfrak{m'}$ and  ideal $J\subset
R'$ such that $z\in J$, $zR'\neq J$, the natural projection $R'/zR'\to R'/J$
induces a trivial map $0:\Tor_n(R'/zR', M)\to \Tor_n(R'/J, M)$;
  \item[(iii)] $\Tor_n(R'/zR', M)\neq 0$ for any non trivial element $z\in \mathfrak{m'}$.
\end{itemize}
\end{lemma}
\begin{proof} (i): This is a special case of Lemma \ref{llemma}.

(ii):  Let $I=(0: z)\subset R'$. Using Lemma \ref{lemma6}, we obtain that
$(0:I)=zR'$. This implies that $(0:I)\subset J$ and $(0:I)\neq J$. Hence there
exists
$y\in I$ such that $(0:y)\subset J$ and $(0:y)\neq J$. Thus we have the inclusions
$zR' \subset (0:y)\subset J$ which give rise to the natural projections
$R'/zR' \to R'/(0:y)\to R'/J$. Using Lemma's \ref{lemma3}, \ref{lemma4} and the fact that $R'/(0:y)\cong yR' \subset R'$, we
obtain that $\Tor_n(R'/(0:y), M)=0$. Hence the composition of the following maps
$\Tor_n(R'/zR', M) \to \Tor_n(R'/(0:y), M)\to \Tor_n(R'/J, M)$
is trivial.

(iii): By (i), we have a non trivial element $x\in \mathfrak{m'}$ such that $\Tor_n(R'/xR', M)\neq 0$.
For any non trivial element $z\in \mathfrak{m'}$, either $z\in xR'$ or
$x\in zR'$. If $z\in xR'$ and $z\neq x$, then there exists $\lambda \in
\mathfrak{m'}$ such that $z=\lambda x$. Define a map $\alpha : R'/xR'\to R'/zR'$ by
$\alpha(r+xR')= \lambda r+zR'$ for all $r\in R'$. Since $x\notin (0:\lambda)$, it
follows that $(0:\lambda)\subset xR'$. This shows that $\alpha$ is injective. Using
Lemma \ref{lemma3}, we obtain that $\alpha$ induces an inclusion $\Tor_n(R'/xR',
M)\hookrightarrow \Tor_n(R'/zR', M)$. Thus $\Tor_n(R'/zR', M)\neq 0$.

In the case when $x\in zR'$ and $x\neq z$, there exists $\lambda'\in \mathfrak{m'}$
such that $x=\lambda' z$. Define a map $\sigma : R'/zR'\to R'/xR'$ by
$\sigma(r+zR')=r +xR'$ for all $r\in R'$. Since $z\notin (0:\lambda')$, we obtain
that $(0:\lambda')\subset zR'$. Thus $\sigma$ is injective. Observe that $xR'\subset \lambda' R'$ and
$xR'\neq \lambda' R'$. Now consider the short exact sequence
 \[0 \rightarrow R'/zR'\xrightarrow{ \ \sigma \ } R'/xR'\xrightarrow{ \ \tau \ }
R'/\lambda' R' \rightarrow 0\,,\]
where $\tau $ is the natural projection. Using $(ii)$, we see that $\tau$ induces the trivial map \newline
$0:\Tor_n( R'/xR', M)\to \Tor_n( R'/\lambda'R', M)$. Therefore $\sigma$ induces an
epimorphism \newline $\Tor_n( R'/zR', M)\twoheadrightarrow \Tor_n( R'/xR', M)$,
which implies that $\Tor_n( R'/zR', M)\neq 0$.
\end{proof}

Let $\deg(r)$ denote the degree of nilpotency of $r\in R$. Noting that the
nilpotency degree of an element $r\in R$ is the smallest $k\in \mathbb{N}$ such that
$r^k=0$, we state our next lemma.

\begin{lemma}\label{deg}
Let $(R, \mathfrak{m})$ be a local Gaussian ring with nilradical $\mathfrak{m}$ and
let $\lambda\in \mathfrak{m}$. If $\mathfrak{m}$ is not nilpotent, then there exists
$z\in \mathfrak{m}$ such that $\deg(z) > \deg(\lambda)$.
\end{lemma}
\begin{proof}
Let $\deg(\lambda)=n$. Suppose the lemma is not true, then $\deg(z)\leq
\deg(\lambda)$ for all $z\in \mathfrak{m}$. We will show that
$\mathfrak{m}^n=0$. This will give us a contradiction, as $\mathfrak{m}$ is not
nilpotent. Towards that end, let $z_1,\ldots,z_n\in \mathfrak{m}$ and consider
$I=(z_1,\ldots,z_n)$. Using Theorem \ref{thm2.2}$(ii)$, we can write $z_i=r_iz+d_i$
for some $z\in I$, $r_i\in R$ and $d_i\in I\cap (0:I)\subset (0:z_i)$ for all $1\leq
i\leq n$. So
\begin{equation}\label{prod}
z_1z_2\cdots z_n=\prod_{i=0}^n (zr_i+d_i)
\end{equation}
After expanding the right hand side of (\ref{prod}), observe that every term of the
expansion except the term $d_1\cdots d_n$ contains a $z$ and some $d_i$, where
$1\leq i\leq n$. By Theorem \ref{thm2.3}$(i)$, it follows that $d_1\cdots d_n=0$.
Since $d_iz=0$ for all $1\leq i\leq n$, every term in the expansion is zero. Thus
$\mathfrak{m}^n=0$, a contradiction.
\end{proof}

\begin{lemma}\label{lemma8}
Let $(R,\mathfrak{m})$ be a local Gaussian ring with nilradical $\mathfrak{m}$
and let $M$ be a $R$-module. If $\wgd_R (M)=n\geq 1$ and $\mathfrak{m}$ is not
nilpotent, then $\Tor_n(R'/a\mathfrak{m'}, M)=0$ for all non trivial $a\in
\mathfrak{m'}$.
\end{lemma}
\begin{proof} If $a\mathfrak{m'}=0$, then Lemma \ref{lemma4} gives the desired
result. So assume that $a\mathfrak{m'}\neq 0$.
We claim that $a\mathfrak{m'}$ is not a finitely generated ideal. Suppose
$a\mathfrak{m'}$ is a finitely generated ideal. Since $R'$ is a local arithmetical ring,
there exists an element $\lambda \in \mathfrak{m'}$ such that $a\mathfrak{m'} =
a\lambda R'$. Let $\deg(x)$ denote
the degree of nilpotency of $x$ for all $x\in \mathfrak{m'}$. Since $\mathfrak{m}$
is not nilpotent, $\mathfrak{m'}$
 is not nilpotent. By Lemma \ref{deg}, there exists $z\in \mathfrak{m'}$ such that
$\deg(z)>\deg(\lambda)$.
  Observe that $\lambda \in z\mathfrak{m'}$, i.e. $\lambda =zh$ for some $h\in
\mathfrak{m'}$. Hence $az\neq 0$. Furthermore $1-hr$ is a unit for all $r\in R'$.
This implies that
$az-a\lambda r=az(1-hr)\neq 0$. Thus $az\notin a\lambda R'$, a contradiction.

Now let $\mathcal{X}$ be the following class of ideals: $J\in \mathcal{X}$ iff
$J\subset a\mathfrak{m'}$ and $J$ is finitely generated. Then
$\Tor_n(R'/a\mathfrak{m'}, M)= \varinjlim_{J\in \mathcal{X}} \Tor_n(R'/ J, M)$.
Since $R'$ is a local arithmetical ring, there exists $c\in \mathfrak{m'}$ such that
$I=cR'$ for all $I\in \mathcal{X}$. As $a\mathfrak{m'}$ is not finitely generated,
$I\neq a\mathfrak{m'}$. Using Lemma \ref{lemma7}$(ii)$, we obtain that the natural
projection $R'/I\to R'/a\mathfrak{m'}$ induces a trivial homomorphism $0:\Tor_n(R'/
I, M)\to \Tor_n(R'/a\mathfrak{m'}, M)$.
Thus the canonical homomorphism
$\Tor_n(R'/ I, M)\to \varinjlim_{J\in \mathcal{X}}\Tor_n(R'/ J, M)$
is trivial for all $I\in \mathcal{X}$. This implies that $\varinjlim_{J\in
\mathcal{X}}\Tor_n(R'/ J, M)=0$.
\end{proof}

Now we are ready to prove the main theorem of this section.
\begin{theorem}\label{thm1}
Let $(R,\mathfrak{m})$ be a local Gaussian ring with nilradical $\mathfrak{m}$. If
$\mathfrak{m}$ is not nilpotent, then $\wgd (R)=\infty$.
\end{theorem}
\begin{proof} Suppose the theorem is not true, then the $\wgd (R)=n<\infty$. Using
Proposition \ref{prop1.1}, we obtain that $n\geq 3$. Let $M$ be a $R$-module with $\wgd_R(M)=n$.
We divide the proof into two cases.

{\bf Case 1.} $R$ does not admit Property (\ref{hypo1}).

Hence there exists $x\in D\setminus 0$ and $a\in R\setminus D$ such that
$(0:x)=aR+D$. Thus we have an isomorphism $R/(aR+D) \cong xR$. Using Lemma
\ref{lemma3}  and noting that $xR\subset R$, we obtain an inclusion $\Tor_n
(R/(aR+D), M)\hookrightarrow \Tor_n (R, M)$. Hence $\Tor_n (R/(aR+D), M)=0$. But
using Lemma \ref{lemma7}$(iii)$, we obtain that $\Tor_n (R/(aR+D), M)\neq 0$, a
contradiction.

{\bf Case 2.} $R$ admits Property (\ref{hypo1}).

Consider the short exact sequence $0\to aR'/a\mathfrak{m'}\to R'/a\mathfrak{m'} \to
R'/aR'\to 0$ where $a\in \mathfrak{m}\setminus D$. From the corresponding long exact
sequence of $\Tor$ groups, consider the following segment
 \[\Tor_n(R' / a\mathfrak{m'} , M) \to \Tor_{n}(R' / aR', M) \to \Tor_{n-1}(aR' /
a\mathfrak{m'},M)\,.\]
  Applying Lemma \ref{lemma8}, we obtain that  $\Tor_n(R' / a\mathfrak{m'} ,M)=0$.
Observing that $aR' / a\mathfrak{m'}\cong R/\mathfrak{m}$ and using Lemma
\ref{lemma1.2} yields
$\Tor_{n-1}(aR' / a\mathfrak{m'}, M)=0$. Hence $\Tor_{n}(R' / aR', M)=0$. But using
Lemma \ref{lemma7} $(iii)$, we obtain that $\Tor_{n}(R' / aR', M)\neq 0$, a contradiction.
\end{proof}

\begin{Remark}
From the proof of Theorem \ref{thm1}, we see that for any module $M$, if $\wgd_R(M)\geq 3$,
then $\wgd_R(M)=\infty$. In particular, $\wgd_R(R/aR)=\infty$ for all $a\in \mathfrak{m}
\setminus D$.
\end{Remark}

\section{Bazzoni-Glaz Conjecture}
In this section, we restate  \cite[Theorem 6.4]{BG} with an additional hypothesis
and prove the theorem under this additional hypothesis. We also give an example to
show that the proof of Theorem 6.4 as given in \cite{BG} is not complete. We need
the next lemma to give a proof of the modification of \cite[Theorem 6.4]{BG}. We can
use the same idea to give a proof of Theorem \ref{main}.

\begin{lemma}\label{lemma5.1}
Let $R$ be a local Gaussian ring with nilradical $\mathcal{N}$. If $\mathcal{N}\neq
D$, then the maximal ideal of $R_{\mathcal{N}}$ is non-zero.
\end{lemma}
\begin{proof}
Using Theorem \ref{thm2.1}, it follows that the nilradical $\mathcal{N}$ is the
unique minimal prime ideal of $R$. Thus the maximal ideal and the nilradical of
$R_{\mathcal{N}}$ coincide and let us denote it by $\mathcal{N'}$. We want to show
that $\mathcal{N'}\neq 0$. Towards that end, let $x\in \mathcal{N}\setminus D$. We
will show that $0\neq \frac{x}{1}\in R_\mathcal{N}$. Suppose not, then there exists
$y\in R\setminus \mathcal{N}$ such that $xy=0$. Using Theorem \ref{thm2.2}$(iv)$, it
follows that $x^2=0$ or $y^2=0$, a contradiction.
\end{proof}

Noting that the nilpotency degree of an ideal $I$ of $R$ is the smallest $k\in
\mathbb{N}$ such that $I^k=0$, we now restate and prove Theorem 6.4 of \cite{BG}.

\begin{theorem}\label{thm5.1} Let $R$ be a Gaussian ring admitting a maximal ideal
$\mathfrak{m}$ such that the nilradical $\mathcal{N}$ of the localization
$R_{\mathfrak{m}}$ is a non-zero nilpotent ideal. If the nilpotency degree of
$\mathcal{N}\geq 3$, then $\wgd(R)=\infty$.
\end{theorem}
\begin{proof}
Let $\mathfrak{m}$ be a maximal ideal of $R$ such that $R_{\mathfrak{m}}$ has a
non-zero nilpotent nilradical $\mathcal{N}$. Using Theorem \ref{thm2.1}, it follows
that $\mathcal{N}$ is the unique minimal prime ideal. Recall that the
Nilradical($S^{-1}R$)=$S^{-1}$(Nilradical($R$)) for any multiplicative closed set
$S\subset R$. Hence $\mathcal{N}=S^{-1}\mathfrak{n}$, where $\mathfrak{n}$ is the
nilradical of $R$ and $S=R\setminus \mathfrak{m}$ is a multiplicatively closed set
in $R$. Furthermore $\mathfrak{n}$ is a prime ideal of $R$. It is clear that the
maximal ideal of $R_{\mathfrak{n}}$ is nilpotent. It follows from Lemma
\ref{lemma5.1} that the maximal ideal of $R_{\mathfrak{n}}$ is non-zero. The rest of
the proof follows \cite[Theorem 6.4]{BG} mutatis mutandis.
\end{proof}
\begin{Remark}
The hypotheses that the nilpotency degree of $\mathcal{N}\geq 3$ in the above
Theorem ensures that $\mathcal{N}\neq D$.
\end{Remark}

We now give an example to show that the hypothesis on the nilpotency degree in
Theorem \ref{thm5.1} is necessary for the conclusion of Lemma \ref{lemma5.1} to
hold.

\begin{Example}\label{EG}
let $\mathbf{k}$ be a field and $\mathbf{k}[X,Y]$ be a polynomial ring in two
variables. Consider a set $S\subset \mathbf{k}[X,Y]/(XY, Y^2)$ defined by
\[
S=\{a+bY+a_1X+ a_2X^2+\cdots+a_nX^n , \;a,b,a_i\in \mathbf{k}, a\neq 0, n\geq 0\}.
\]
Then $S$ is multiplicative set in $\mathbf{k}[X,Y]/(XY, Y^2)$. Define $R=S^{-1}\big(
\mathbf{k}[X,Y]/(XY, Y^2)\big)$. It is easy to see that the unique maximal ideal of
$R$ is given by $\mathfrak{m}=\{xf(x)+b_1y \mid f(x)\in \mathbf{k}[x]\; and\; b_1\in
\mathbf{k}\}$, where $x, y$ are the images of $X, Y$ in $R$. Any $c,d\in
\mathfrak{m}$ has the form, $c=\lambda_1 y+c_1x+c_2x^2+\cdots c_nx^n$ and
$d=\lambda_2 y+d_1x+d_2x^2+\cdots +d_mx^m$, where $m,n\in \mathbb{N}$ and $c_i,
d_j\in \mathbf{k}$ for $1\leq i\leq n$ and $1\leq j\leq m$. Let $i,j\in \mathbb{N}$
be such that $c_i\ne 0$ and $d_j\neq 0$ and let $i,j$ be the least natural numbers
with this property. Then one can rewrite $c,d$ as $c=\lambda_1
y+x^i(c_i+c_{i+1}'x+\cdots+c_{n}'x^{n-i})$ and $d=\lambda_2
y+x^j(d_j+d_{j+1}'x+\cdots+d_{m}'x^{m-j})$, where $c_{t}'=c_{t+1}c_{i}^{-1}$ and
$d_{s}'=d_{s+1}d_{j}^{-1}$ for all $i\leq t\leq n-i$ and $j\leq s\leq n-j$. Observe
that $c_i+c_{i+1}'x+\cdots+c_{n}'x^{n-i}$ and $d_j+d_{j+1}'x+\cdots+d_{m}'x^{m-j}$
are units in $R$. Furthermore $(c,d)^2=(c^2,cd,d^2)$, $c^2=x^{2i}u_1^2$,
$d^2=x^{2j}u_{2}^2$ and $cd=x^{i+j}u_1u_2$, where $u_1,u_2\notin \mathfrak{m}$. Now
using Theorem \ref{thm2.2}$(iv)$, it can be verified that $R$ is a local Gaussian
ring. Its nilradical $\mathfrak{n}=(y)\subset R$ is not trivial, while the
nilradical of $R_\mathfrak{n}$ is trivial as $\frac{y}1=\frac{xy}x=0$.
\end{Example}

\begin{theorem}\label{main}
Let $R$ be a non-reduced local Gaussian ring with nilradical $\mathcal{N}$. If
$\mathcal{N}\neq D$, then $\wgd (R)=\infty$.
\end{theorem}
\begin{proof} By Theorem \ref{thm2.1}, the nilradical $\mathcal{N}$ is the unique
minimal prime ideal. Thus the maximal ideal and the nilradical of $R_{\mathcal{N}}$
coincide and let us denote it by $\mathcal{N'}$. Since $\wgd(R)\geq
\wgd(R_{\mathcal{N}})$, it suffices to show that $\wgd (R_\mathcal{N})=\infty$.
Using Lemma \ref{lemma5.1}, it follows that $\mathcal{N'}\neq 0$. Hence we have a
local Gaussian ring $(R_\mathcal{N}, \mathcal{N'})$ with $\mathcal{N'}\neq 0$. If
$\mathcal{N'}$ is nilpotent, then Proposition \ref{prop2.1} implies that $\wgd
(R_\mathcal{N})=\infty$. If $\mathcal{N'}$ is not nilpotent, then Theorem \ref{thm1}
implies that $\wgd (R_\mathcal{N})=\infty$.
\end{proof}

We claim that to prove the Bazzoni-Glaz Conjecture, it remains to consider the case
of a local Gaussian ring with nilradical $\mathcal{N}=D\neq 0$.
Let $R$ be a Gaussian ring (not necessarily local) and let
$\mathcal{N}_{\mathfrak{p}}$ denote the nilradical of $R_{\mathfrak{p}}$ for any
$\mathfrak{p}\in \Spec (R)$. We have the
following cases:

\begin{itemize}
  \item[(i)]  $R_{\mathfrak{p}}$ is domain for all $\mathfrak{p}\in \Spec (R)$;
 \item[(ii)]  there exists a $\mathfrak{p}\in \Spec (R)$ such that the
$\mathcal{N}_{\mathfrak{p}}\neq 0$ and $\mathcal{N}^2_{\mathfrak{p}}\neq 0$;
   \item[(iii)]  there exists a $\mathfrak{p}\in \Spec (R)$ such that
$\mathcal{N}_{\mathfrak{p}}\neq 0$ and $\mathcal{N}^2_{\mathfrak{p}}= 0$.
\end{itemize}

We remind the reader that if $R_{\mathfrak{p}}$ is not a domain, then
$\mathcal{N}_{\mathfrak{p}}\neq 0$ and hence all possible cases are listed above.
In case (i) $\wgd (R)\leq 1$, while in case (ii) $\wgd (R)= \infty$. Hence to prove
the Bazzoni-Glaz Conjecture it remains to show the following.

\begin{theorem}\label{thm5.5}
Let $R$ be a non-reduced local Gaussian ring with nilradical $\mathcal{N}$. If
$\mathcal{N}^2=0$, then $\wgd (R)=\infty$.
\end{theorem}

We prove this in the next section.

\section{Local Gaussian rings with square free nilradical}

Throughout this section $R$ is a local Gaussian ring with maximal ideal $\mathfrak{m}$
and nilradical $\mathcal{N}=D\neq 0$. Without loss of generality, we may assume
that each element of $\mathfrak{m}$ is a zero divisor and $\mathfrak{m}\neq D$.
Since the nilradical is the unique minimal prime ideal in a local Gaussian ring, one easily
checks that $D_\mathfrak{p}$ is also the nilradical of $R_\mathfrak{p}$ for all
$\mathfrak{p}\in Spec (R)$. We show that the assumption $\wgd(R)<\infty$ leads to
a contradiction. By Proposition \ref{prop1.1}, we know that $\wgd (R)\geq 3$.

\begin{lemma}\label{adlemma1}
Let $\wgd (R)=n$. If $n<\infty$, then $\wgd_R (R/D)=n-1$.
\end{lemma}
\begin{proof} By Lemma \ref{llemma}, there exists an $a\in \mathfrak{m}\setminus D$ such that
$\wgd_R(R/(aR+D))=n$. Since $R/D$ is domain, we have the following short exact sequence,
$0\rightarrow R/D\overset{a_{m}}{\rightarrow}R/D\overset{\pi}{\rightarrow} R/(aR+D)\rightarrow 0$,
where $a_{m}$ is multiplication by $a$ and $\pi$ is the natural projection. This implies
$\wgd_R(R/D)\geq \wgd_R (R/(aR+D))-1=n-1$. On the other hand, using Lemma \ref{lemma4}, we obtain that $\wgd_R (R/D)\leq n-1$.
\end{proof}

\begin{lemma}\label{adlemma2}
Let $i\geq 1$. If $a\in \mathfrak{m}\setminus D$ is a zero-divisor on $\Tor_{i-1}(R/D, R/D)$, then
\newline $\Tor_{i}(R/(aR+D), R/D)\neq 0$. Furthermore, if $\wgd R=n<\infty$, then \newline
$\Tor_{n-1}(R/(aR+D), R/D)\neq 0$ if and only if $a$ is a zero-divisor on $\Tor_{n-2}(R/D, R/D)$.
\end{lemma}
\begin{proof} We have a short exact sequence:
$0\rightarrow R/D\overset{a_m}{\rightarrow}R/D\overset{\pi}{\rightarrow} R/(aR+D)\rightarrow 0$,
where $a_m$ is multiplication by $a$ and $\pi$ is the natural projection. Consider
the following segment of the corresponding long exact sequence of Tor groups
$$
\Tor_{i}(R/(aR+D), R/D)\rightarrow \Tor_{i-1}(R/D, R/D) \xrightarrow{a_m^{*}} \Tor_{i-1}(R/D, R/D),
$$
where $a_m^{*}$ is multiplication by $a$. This exact sequence implies the first part of the lemma.
If $\wgd R=n<\infty$, then first using Lemma \ref{adlemma1} and then using Lemma \ref{lemma4}, we obtain that $\Tor_{n-1}(R/D, R/D)=0$.
Hence the above exact sequence for $i=n-1$ implies that \newline $\Tor_{n-1}(R/(aR+D), R/D)=
\Ker \{a_m^{*} : \Tor_{n-2}(R/D, R/D)\rightarrow \Tor_{n-2}(R/D, R/D)\}$.
This completes the proof.
\end{proof}

 The following easy observation will be used several times in the sequel.

\begin{lemma}\label{adlemma3}
Let $M$ be a $R$-module. If $x\in M$, then the  set
$I=\{a\in \mathfrak{m}| a^qx=0 \;\text{for some}\; q\in \mathbb{N} \}$ is a prime ideal of $R$.
\end{lemma}
\begin{proof} Set $J=\{a\in \mathfrak{m} , ax=0 \}$. Observe that $J$ is an ideal of $R$. It can be easily verified that $I/J$
is the nilradical of $R/J$. Since $R/J$ is a local Gaussian ring, $I/J$ is prime ideal and so is $I$.
\end{proof}

\begin{lemma}\label{adlemma4} Let $\wgd R=n$. If $n<\infty$, then there exist a $\mathfrak{p}\in Spec (R)$ such that the following conditions hold.
\begin{itemize}
  \item[(i)] $D_\mathfrak{p}\neq \mathfrak{p}_\mathfrak{p}$ and there exist
$\omega\in \Tor^{R_\mathfrak{p}}_{n-2}(R_\mathfrak{p}/D_\mathfrak{p}, R_\mathfrak{p}/D_\mathfrak{p})$
such that $\omega\neq 0$ and for all $c\in \mathfrak{p}_\mathfrak{p}$,
$c^q\omega =0$ for some $q\in \mathbb{N}$ ;
 \item[(ii)] for all $c\in \mathfrak{p}_\mathfrak{p}\setminus D_\mathfrak{p}$,
 $\Tor^{R_\mathfrak{p}}_{n-1}(R_\mathfrak{p}/(cR_\mathfrak{p}+D_\mathfrak{p}), R_\mathfrak{p}/D_\mathfrak{p})\neq 0$;
 \item[(iii)] each element of $\mathfrak{p}_\mathfrak{p}$ is a zero-divisor;
 \item[(iv)] $\wgd(R_\mathfrak{p})=n$.
 \end{itemize}
\end{lemma}
\begin{proof} (i): By Lemma's \ref{llemma} and \ref{adlemma1}, there exists an  $a\in \mathfrak{m}\setminus D$
such that $\Tor_{n-1}(R/(aR+D), R/D)\neq 0$. Using Lemma \ref{adlemma2}, there exist a $\omega\in \Tor_{n-2}(R/D, R/D)$
such that $\omega\neq 0$ and $a\omega =0$. Set
$\mathfrak{p}= \{b\in \mathfrak{m} , b^q \omega=0 \;\text{for some}\; q\in \mathbb{N} \}$.
By Lemma \ref{adlemma3}, it follows that $\mathfrak{p}\in Spec (R)$. Moreover $a\in \mathfrak{p}$ and since $D$ is a prime ideal, $\frac{a}{1}\notin D_\mathfrak{p}$.
Hence $D_\mathfrak{p}\neq \mathfrak{p}_\mathfrak{p}$.
We claim that $\omega=\frac{\omega}{1}\in \Tor_{n-2}(R/D, R/D)_\mathfrak{p} =
\Tor^{R_\mathfrak{p}}_{n-2}(R_\mathfrak{p}/D_\mathfrak{p}, R_\mathfrak{p}/D_\mathfrak{p})$
is not trivial. If this is not the case, then there exist $b\in R\setminus \mathfrak{p}$ such that $b \omega =0$,
a contradiction. By definition, for any $c\in \mathfrak{p}_\mathfrak{p}$, there exist
$q\in \mathbb{N}$ such that $c^q\omega =0$.

(ii): By (i), $c\in \mathfrak{p}_\mathfrak{p}\setminus D_\mathfrak{p}$ is a zero-divisor on
$\Tor^{R_\mathfrak{p}}_{n-2}(R_\mathfrak{p}/D_\mathfrak{p}, R_\mathfrak{p}/D_\mathfrak{p})$.
Hence Lemma \ref{adlemma2} implies (ii).

(iii): If this is not the case, then there exists $c\in \mathfrak{p}_\mathfrak{p}\setminus D_\mathfrak{p}$
which is regular in $R_\mathfrak{p}$. Using Theorems \ref{thm2.1}(ii) and \ref{thm2.3}(ii), we obtain that
$D_\mathfrak{p} \subset cR_\mathfrak{p}+(0:c)=cR_\mathfrak{p}$. This implies that $cR_\mathfrak{p} =
cR_\mathfrak{p} + D_\mathfrak{p}$.
By (ii), $\Tor^{R_\mathfrak{p}}_{n-1}(R_\mathfrak{p}/cR_\mathfrak{p}, R_\mathfrak{p}/D_\mathfrak{p})\neq 0$.
On the other hand, we have the following free resolution of $R_\mathfrak{p}/cR_\mathfrak{p}$:
$0\rightarrow R_\mathfrak{p}\overset{c_m}{\rightarrow}R_\mathfrak{p}\overset{\pi}{\rightarrow}
R_\mathfrak{p}/cR_\mathfrak{p}\rightarrow 0$, where $c_m$ is multiplication by $c$ and
$\pi$ is the natural projection. This implies that $\wgd_{R_\mathfrak{p}}(R_\mathfrak{p}/cR_\mathfrak{p})\leq 1$,
but this is a contradiction since $n\geq 3$.

(iv): It suffices to check that $\wgd (R_\mathfrak{p})\geq n$. By (ii),
$\wgd_{R_\mathfrak{p}}(R_\mathfrak{p}/D_\mathfrak{p})\geq n-1$. Using Lemma \ref{adlemma1}, we obtain that
$\wgd (R_\mathfrak{p})\geq n$.
\end{proof}

Thus, if the Bazzoni-Glaz Conjecture is not true, then Lemma \ref{adlemma4} enables us to assume
the existence of a local Gaussian ring $(R, \mathfrak{m})$ with $\wgd(R)=n<\infty$,
satisfying the following conditions:

(C1) $D\neq \mathfrak{m}$ and $\Tor_{n-1}(R/(aR+D), R/D)\neq 0$ for all $a\in \mathfrak{m}\setminus D$.

\noindent Moreover, since $n\geq 3$, $\Tor_{n-2}(R/D, R/D)=\Tor_{n-3}(D, R/D)$ and
we also can assume that $R$ satisfies the following:

(C2) there exists an $\omega\in \Tor_{n-3}(D, R/D)$ such that $\omega\neq 0$ and for all
$a\in \mathfrak{m}$, $a^q\omega =0$ for some $q\in \mathbb{N}$.

\begin{lemma}\label{adlemma5} Let $\wgd (R)=n < \infty$. If $R$ satisfies (C1), then the following conditions hold.
\begin{itemize}
\item[(i)] $\mathfrak{m}$ is flat;
\item[(ii)] if $I=\mathfrak{m}, D$ or $(0:c)$ for any $c\in \mathfrak{m}$, then $I=I\mathfrak{m}$.
 Moreover, if $x\in I$, then there exist $x'\in I$ and $a\in \mathfrak{m}$ such that $x=ax'$;
 \item[(iii)] $\mathfrak{m}$ is not finitely generated modulo $D$.
 \end{itemize}
\end{lemma}
\begin{proof} (i): As a result of Lemma \ref{lemma1.2}, it suffices to check that $R$ admits Property \ref{hypo1}.
If this is not the case, then there exists a $d\in D\setminus 0$ and $a\in \mathfrak{m}\setminus D$
such that $(0:d)=aR+D$. This implies that $R/(aR+D) \cong dR$. First using Lemma \ref{adlemma1} and then using Lemma \ref{lemma3}, we obtain that $\Tor_{n-1}(dR, R/D)= 0=\Tor_{n-1}(R/(aR+D), R/D)$, contradicting (C1).

(ii): By (i), we have that $\wgd_R(R/\mathfrak{m})=1$. Using Lemma's \ref{lemma5}, \ref{lemma4}
and \ref{lemma3} respectively with $I=\mathfrak{m}, D$ and $(0:c)$, we obtain that $\Tor_1(R/I,R/\mathfrak{m})=0$. Thus $I=I\mathfrak{m}$ because $\Tor_1(R/I,R/\mathfrak{m})=I/I\mathfrak{m}$. Consider any $x\in I$. We have $x=x_1a_1+\cdots+x_sa_s$, where $x_1,\cdots, x_s\in I$ and $a_1,\cdots,a_s\in \mathfrak{m}$.
Let $J$ be an ideal generated by $x_1,\cdots, x_s,a_1,\cdots,a_s$. By Theorem \ref{thm2.2}$(ii)$, there exists an
$a\in J$ such that $J=aR+(0:J)$. Set $a_i=ar_i+\lambda_i$, where $r_i\in R$ and
$\lambda_i\in (0:J)$ for all $1\leq i\leq s$. Then $x=x_1ar_1+\cdots +x_sar_s=ax'$ where
$x'=x_1r_1+\cdots +x_sr_s\in I$.

(iii): If this is not the case, then there exists an $a\in \mathfrak{m}\setminus D$ such that
$\mathfrak{m}=aR+D$. Flatness of $\mathfrak{m}$ implies that $\mathfrak{m}=\mathfrak{m}^2$. So $\mathfrak{m}=a^2 R + D$.
Hence $a = a^2r+d$, where $r\in R$ and $d\in D$. Thus $a(1-ar)\in D$, which implies that $a\in D$,
a contradiction.
\end{proof}

\begin{lemma}\label{adlemma6} Let $\wgd(R)=n<\infty$. If $R$ satisfies (C1) and (C2), then there exist
$\overline{b}\in \mathfrak{m}\setminus D$ with the following property:
for all $b\in \mathfrak{m}$ with $\overline{b}\in bR+D$, there exists $\omega_b\in \Tor_{n-3}(D, R/D)$
such that $b\omega_b\neq 0$ and $b^q\omega_b =0$ for some $q\in \mathbb{N}$.
\end{lemma}
\begin{proof} If $n=3$, then $\Tor_{n-3}(D, R/D)=D\otimes (R/D)=D$. Thus we are given
$\omega\in D\setminus 0$ with the following property: for all $a\in \mathfrak{m}$
there exist $q\in \mathbb{N}$ such that $a^q\omega =0$. By Lemma \ref{adlemma5}$(ii)$,
there exists $\overline{b}\in \mathfrak{m}$ and $\overline{\omega}\in D$ such that
$\omega=\overline{b}\overline{\omega}$. Take any $b\in \mathfrak{m}$
such that $\overline{b}\in bR+D$. Set $\overline{b}=br+d$ where $r\in R$ and $d\in D$ and
let $\omega_b = r\overline{\omega}$. Then $b\omega_b=br\overline{\omega}=(br+d)\overline{\omega}=
\overline{b}\overline{\omega}=\omega\neq 0$. Moreover, according to (C2), there exists
$q\in \mathbb{N}$ such that $b^q\omega =0$, implying $b^{q+1}\omega_b=0$.
Hence the lemma is proved for $n=3$.

Set $n\geq 4$. Consider a free resolution of $R/D$:
$\cdots\xrightarrow{\partial_2}R^{X_1}\xrightarrow{\partial_1}R^{X_0}\xrightarrow{\partial_0} R/D$,
where $X_i$ are sets. Set $K =\Ker \partial_{n-4}$. Let $\sigma : K \rightarrow R^{X_{n-4}}$
be the natural inclusion and $\overline{\sigma}:K\otimes D \rightarrow D^{X_{n-4}}$
be the natural homomorphism obtained after tensoring by $D$.
Then $\Tor_{n-3}(D, R/D)=\Ker \overline{\sigma}$. Let $\omega$ be as in (C2).
Set $\omega = x_1\otimes d_1+\cdots+x_s\otimes d_s\in \Ker \overline{\sigma}$,
where $x_i\in K$ and $d_i\in D$ for all $1\leq i\leq s$. By Lemma \ref{adlemma5}$(ii)$,
there exists $a_i\in \mathfrak{m}\setminus D$ such that $d_i\in a_iD$ for all $1\leq i\leq s$.
Since $R/D$ is a local arithmetical ring, there exists an $a\in \{a_1,\cdots, a_s\}$ such that
$a_i\in aR+D$ for all $1\leq i\leq s$. Hence $d_i\in aD$ for all $1\leq i\leq s$.
Set $d_i=ad_i'$, where $d_i'\in D$ for all $1\leq i\leq s$ and
set $\omega'=x_1\otimes d_1'+\cdots+x_s\otimes d_s'$.
Observe that $a\omega'=\omega$, which implies that $a\overline{\sigma}(\omega')=0$. We have
$\overline{\sigma}(\omega')\in D^{X_{n-4}}$. Let $\lambda_1, \cdots, \lambda_t\in D$
be the finitely many non zero entries of $\overline{\sigma}(\omega')$.
Then $a\in (0:\lambda_i)$ for all $1\leq i\leq t$. By Lemma \ref{adlemma5}$(ii)$,
there exist $c_i\in (0:\lambda_i)$ such that $a\in c_i\mathfrak{m}$ for all
$1\leq i\leq t$. Since $R/D$ is a local arithmetical ring, there exist a
$c\in \{c_1,\cdots, c_t\}$ such that $c\in c_iR+D$ for all $1\leq i\leq t$.
Thus $c\in (0:\lambda_i)$ for all $1\leq i\leq t$ and $a\in c\mathfrak{m}$.
Set $\overline{\omega}=c\omega'$ and $a=c\overline{b}$, where $\overline{b}\in \mathfrak{m}\setminus D$.
Observe that $\overline{\omega}\in \Ker \overline{\sigma}$ and $\overline{b}\overline{\omega}=\omega\neq 0$.
Choose any $b\in \mathfrak{m}$ such that $\overline{b}\in bR+D$. Set $\overline{b}=br+d$
where $r\in R$ and $d\in D$ and set $\omega_b = r\overline{\omega}$. Then
$b\omega_b=br\overline{\omega}=(br+d)\overline{\omega}=\overline{b}\overline{\omega}=
\omega\neq 0$. Moreover, according to (C2), there exists $q\in \mathbb{N}$
such that $b^q\omega =0$, implying that $b^{q+1}\omega_b=0$.
\end{proof}

\begin{lemma}\label{adlemma7} Let $\wgd(R)=n<\infty$. If $R$ satisfies (C1) and (C2), then there exists a
$\mathfrak{p}\in Spec(R)$ such that the following conditions hold.

\begin{itemize}
\item[(i)] $D_\mathfrak{p}\neq \mathfrak{p}_\mathfrak{p}$ and each element of $\mathfrak{p}_\mathfrak{p}$
is a zero divisor. Furthermore, there exist $d\in D$ such that $\frac{d}{1}\in D_\mathfrak{p}$ is not
trivial and for any $c\in \mathfrak{p}_\mathfrak{p}$, $\frac{d}{1}c^q=0$ for some $q\in \mathbb{N}$;
\item[(ii)] for all $c\in \mathfrak{p}_\mathfrak{p}\setminus D_\mathfrak{p}$,
$\Tor^{R_\mathfrak{p}}_{n-1}(R_\mathfrak{p}/(cR_\mathfrak{p}+D_\mathfrak{p}), R_\mathfrak{p}/D_\mathfrak{p})\neq 0$;
\item[(iii)] $\wgd(R_\mathfrak{p})=n$.
\end{itemize}
\end{lemma}
\begin{proof}(i): Choose $\overline{b}\in \mathfrak{m}\setminus D$ as in Lemma \ref{adlemma6}.
Fix a non trivial element $d\in (0:\overline{b})$.  Set
$\mathfrak{p}= \{a\in \mathfrak{m} , a^q d=0 \;\text{for some}\; q\in \mathbb{N} \}$.
Using Lemma \ref{adlemma3}, it follows that $\mathfrak{p}\in Spec (R)$ and note that $\overline{b}\in \mathfrak{p}$.
Since $D$ is a prime ideal, $\frac{\overline{b}}{1}\notin D_\mathfrak{p}$. Hence
$D_\mathfrak{p}\neq \mathfrak{p}_\mathfrak{p}$. To prove the remaining part of (i), it suffices
to check that $\frac{d}{1}\neq 0$. If $\frac{d}{1}= 0$, then there exists an
$a\in R\setminus \mathfrak{p}$ such that $ad=0$, a contradiction.

(ii): Using Lemma \ref{adlemma2}, it suffices to check that any element
$c\in \mathfrak{p}_\mathfrak{p}\setminus D_\mathfrak{p}$ is a zero-divisor on
$\Tor^{R_\mathfrak{p}}_{n-2}(R_\mathfrak{p}/D_\mathfrak{p}, R_\mathfrak{p}/D_\mathfrak{p})$.
This is the same as showing that $c$ is a zero-divisor on
$\Tor^{R_\mathfrak{p}}_{n-3}(D_\mathfrak{p}, R_\mathfrak{p}/D_\mathfrak{p})$.
It suffices to consider the case $c=\frac{b}{1}$ where $b\in \mathfrak{p}\setminus D$.
Since $R/D$ is a local arithmetical ring, either $\overline{b}\in bR+D$ or $b\in \overline{b}R+D$, where $\overline{b}$ is defined as in (i).

{\bf Case 1.} $\overline{b}\in bR+D$.

By Lemma \ref{adlemma6}, there exists $\omega_b\in \Tor_{n-3}(D, R/D)$ such that
$b\omega_b\neq 0$ and $b^q\omega_b =0$ for some $q\in \mathbb{N}$. We claim that
$\frac{\omega_b}{1}\in \Tor_{n-3}(D, R/D)_\mathfrak{p} =\Tor^{R_\mathfrak{p}}_{n-3}(D_\mathfrak{p}, R_\mathfrak{p}/D_\mathfrak{p})$
is not trivial. If this is not the case, then there exists an $a\in R\setminus \mathfrak{p}$ such that
$a\omega_b=0$. Since $b\in \mathfrak{p}$, we obtain that $b\in aR+D$. Thus $b\omega_b=0$,
a contradiction. Hence we have a nontrivial element
$\frac{\omega_b}{1}\in \Tor^{R_\mathfrak{p}}_{n-3}(D_\mathfrak{p}, R_\mathfrak{p}/D_\mathfrak{p})$
such that $\frac{\omega_b}{1}b^q =0$ for some $q\in \mathbb{N}$. This shows that $b$
is a zero-divisor on $\Tor^{R_\mathfrak{p}}_{n-3}(D_\mathfrak{p}, R_\mathfrak{p}/D_\mathfrak{p})$.

{\bf Case 2.} $b\in \overline{b}R+D$.

Set $b=\overline{b}r+\lambda$, where $r\in R$ and $\lambda\in D$. As in the previous case, there exists
$\overline{\omega}\in \Tor^{R_\mathfrak{p}}_{n-3}(D_\mathfrak{p}, R_\mathfrak{p}/D_\mathfrak{p})$
such that $\overline{\omega}\neq 0$ and $\overline{b}\overline{\omega}=0$. Notice that
$b \overline{\omega} = \overline{b}r \overline{\omega}=0$, proving that $b$ is a zero-divisor on
$\Tor^{R_\mathfrak{p}}_{n-3}(D_\mathfrak{p}, R_\mathfrak{p}/D_\mathfrak{p})$.

(iii): It suffices to check that $\wgd (R_\mathfrak{p})\geq n$. By (ii),
$\wgd_{R_\mathfrak{p}}(R_\mathfrak{p}/D_\mathfrak{p})\geq n-1$. Using Lemma \ref{adlemma1}, we obtain that
$\wgd (R_\mathfrak{p})\geq n$.
\end{proof}

Thus, if the Bazzoni-Glaz conjecture is not true, then using Lemma \ref{adlemma7}, it is possible
to assume the existence of a local Gaussian ring $(R, \mathfrak{m})$ such that $\wgd (R)=n<\infty$
and satisfying the following conditions:

(C1) $D\neq \mathfrak{m}$ and for all $c\in \mathfrak{m}\setminus D$, $\Tor_{n-1}(R/(cR+D), R/D)\neq 0$;

(C3) there exists $d\in D$ such that $d\neq 0$ and for any $c\in \mathfrak{m}$,
$c^q d =0$ for some $q\in \mathbb{N}$.

We need one more Lemma.

\begin{lemma}\label{adlemma8} Let $\wgd (R)=n<\infty$. If $R$ satisfies (C1) and (C3), then there exists an
$a\in \mathfrak{m}\setminus D$ such that $\Tor_{n-1}(R/(a\mathfrak{m}+D), R/D)= 0$.
\end{lemma}
\begin{proof} Choose $d$ as in (C3). By Lemma \ref{adlemma5}$(ii)$, there exists $d'\in D$ and
$c\in\mathfrak{m}\setminus D$ such that $d=cd'$. Using the same lemma, there exists $a,b\in \mathfrak{m}\setminus D$
such that $c=ab$. Let $\mathcal{X}$ be the following class of ideals: $I\in \mathcal{X}$ iff
$aR+bR+D\subset I\subset \mathfrak{m}$ and $I/D$ is finitely generated. Then
$\Tor_{n-1}(R/(a\mathfrak{m}+D), R/D)=\lim_{I\in \mathcal{X}}\Tor_{n-1}(R/(aI+D), R/D)$.
Take any $I_1\in \mathcal{X}$. To prove the lemma it is sufficient to show that the canonical map
\newline $\Tor_{n-1}(R/(aI_1+D), R/D)\rightarrow \lim_{I\in \mathcal{X}}\Tor_{n-1}(R/(aI+D), R/D)$
is trivial. Since $I_1$ is finitely generated modulo $D$, there exist $b_1\in \mathfrak{m}\setminus D$
such that $I_1=b_1R+D$. By Lemma \ref{adlemma5}$(iii)$, $\mathfrak{m}$ is not finitely generated
modulo $D$. Thus there exists $I_2\in \mathcal{X}$ such that $I_1\subset I_2$ and $I_1\neq I_2$.
Set $I_2=b_2R+D$, where $b_2\in \mathfrak{m}\setminus D$. Using Theorem \ref{thm2.3}$(ii)$,
there exist $x\in \mathfrak{m}$ and $d_1\in (0:b_2)$ such that $b_1=xb_2+d_1$. Notice that
$a\in b_2R+D$, implying that $(0:b_2)\subset(0:a)$. Hence $a d_1=0$ and $ab_1=xab_2$.
Consider the following commutative diagram with exact rows:
$$
\xymatrix {
0 \ar[r] & (0:ab_2) \ar[d]^{inclusion}\ar[r] & R \ar[d]^{1_R}\ar[r]^{f_2} &ab_2R \ar[d]^{f_3}\ar[r] & 0\;\;\\
0 \ar[r] & (0:ab_1) \ar[r] & R \ar[r]^{f_1} & ab_1R \ar[r] & 0 \;,}
$$
where $f_1,f_2$ and $f_3$ are defined as multiplications by $ab_1, ab_2$ and $x$, respectively.

Claim: The natural inclusion $(0:ab_2)\subset (0:ab_1)$ is not onto.

If this is not the case, then the first two vertical homomorphisms of the above diagram are
isomorphisms. The short five lemma would imply that $f_3:ab_2R \rightarrow ab_1R$ is also an isomorphism. Since
$b\in I_2$, $b=b_2r+d_2$, where $r\in R$ and $d_2\in D$. So
$d=d'c=d'ab=d'a(b_2r+d_2)=d'ab_2r\in ab_2R$. This implies that $x^id\in ab_2R$ for all $i\geq 0$.
By (C3), there exists $q\in \mathbb{N}$ such that $x^qd=0$ and $x^{q-1}d\neq 0$. Then
$f_3(x^{q-1}d)=x^{q}d=0$. Hence $f_3$ is not injective, a contradiction.

Observing that $ab_1\in \mathfrak{m}\setminus D$, our claim implies that there exist a $\lambda \in D$ such that $\lambda \in (0:ab_1)$ and
$\lambda \notin (0:ab_2)$. This implies that
$ab_1R+D \subset (0:\lambda) \subset ab_2R+D$. Hence $aI_1+D \subset (0:\lambda) \subset aI_2+D$.
Thus we have natural projections $R/(aI_1+D) \rightarrow R/(0:\lambda) \rightarrow R/(aI_2+D)$.
Using Lemma's \ref{lemma3}, \ref{adlemma1} and the fact that $R/(0:\lambda)\cong \lambda R \subset R$, we obtain that $\Tor_{n-1}(R/(0:\lambda) ,R/D)=0$. Hence the natural map $\Tor_{n-1}(R/(aI_1+D) ,R/D)\rightarrow\Tor_{n-1}(R/(aI_2+D) ,R/D)$
is trivial, as it is the composition of the following natural maps
$$
\Tor_{n-1}(R/(aI_1+D) ,R/D)\rightarrow \Tor_{n-1}(R/(0:\lambda) ,R/D) \rightarrow\Tor_{n-1}(R/(aI_2+D) ,R/D).
$$
This implies that $\Tor_{n-1}(R/(aI_1+D), R/D)\rightarrow \lim_{I\in \mathcal{X}}\Tor_{n-1}(R/(aI+D), R/D)$
is also trivial.
\end{proof}

{\bf Proof of Theorem \ref{thm5.5}.} If the theorem is not true, then there exists
a local Gaussian ring $R$ which satisfies (C1) and (C2). Choose $a\in \mathfrak{m}\setminus D$
 prescribed by Lemma \ref{adlemma8}. Consider the following short exact sequence:
$$
0\rightarrow (aR+D)/(a\mathfrak{m}+D)\rightarrow R/(a\mathfrak{m}+D)\rightarrow R/(aR+D)\rightarrow 0.
$$
From the corresponding long exact sequence of Tor groups, consider the segment:
$$
\Tor_{n-1}(R/(a\mathfrak{m}+D), R/D)\rightarrow \Tor_{n-1}(R/(aR+D), R/D)\rightarrow
$$
$$
\rightarrow\Tor_{n-2}((aR+D)/(a\mathfrak{m}+D), R/D).
$$
There is an isomorphism $(aR+D)/(a\mathfrak{m}+D)\cong R/\mathfrak{m}$. Moreover, Lemma
\ref{adlemma5}$(i)$ implies that $\wgd_R(R/\mathfrak{m})=1$. Applying Lemma \ref{lemma4} and noting that $n\geq 3$,
we obtain that
$$\Tor_{n-2}((aR+D)/(a\mathfrak{m}+D), R/D)=0.$$
Now using the above exact sequence we obtain that $\Tor_{n-1}(R/(aR+D), R/D)=0$, but this contradicts (C1).

\end{document}